\documentclass[11pt]{amsart}
\usepackage{amssymb,amsmath,cite,hyperref}
\usepackage[usenames]{color} \newtheorem{theorem}{Theorem}[section]

\newtheorem{proposition}[theorem]{Proposition}
\newtheorem{lemma}[theorem]{Lemma}

\newtheorem{corollary}[theorem]{Corollary}
\newtheorem{example}[theorem]{Example}

\DeclareMathOperator{\lcm}{lcm}

\DeclareMathOperator{\Ima}{Im}
\title{Some results on the subadditivity condition of syzygies}
\author{Abed Abedelfatah}
\address{Department of Mathematics, ORT Braude College, 2161002 Karmiel, Israel}
\email{abed@braude.ac.il}
\begin{document}
\maketitle
\begin{abstract}

Let $S=K[x_1,\ldots,x_n]$, where $K$ is a field, and $t_i$ denotes the maximal shift in the minimal graded free $S$-resolution of the graded algebra $S/I$ at degree $i$. In this paper, we prove:\\
\begin{itemize}
  \item If $I$ is a monomial ideal of $S$ and $a\geq b-1\geq0$ are integers such that $a+b\leq\mathrm{proj~dim}(S/I)$, then $$t_{a+b}\leq t_a+t_1+t_2+\cdots+t_b-\frac{b(b-1)}{2}.$$
  \item If $I=I_{\Delta}$ where $\Delta$ is a simplicial complex such that $\dim(\Delta)< t_a-a$ or $\dim(\Delta)< t_b-b$, then $$t_{a+b}\leq t_a+t_b.$$
  \item If $I$ is a monomial ideal that minimally generated by $m_1,\dots,m_r$ such that $\frac{\lcm(m_1,\dots,m_r)}{\lcm(m_1,\dots,\widehat{m}_i,\dots,m_r)}\notin K$ for all $i$, where $\widehat{m}_i$ means that $m_i$ is omitted, then $t_{a+b}\leq t_a+t_b$ for all $a,b\geq0$ with $a+b\leq\mathrm{proj~dim}(S/I)$.
\end{itemize}
\end{abstract}

\section{Introduction}

Let $S=K[x_1,\ldots,x_n]$, where $K$ is a field and let $I$ be a graded ideal of $S$ and suppose $S/I$ has minimal graded free $S$-resolution
$$0\rightarrow F_p=\bigoplus_{j\in \mathbb{N}}S(-j)^{\beta _{p,j}}\rightarrow\cdots\rightarrow F_1=\bigoplus_{j\in \mathbb{N}}S(-j)^{\beta _{1,j}}\rightarrow F_0=S\rightarrow S/I\rightarrow0.$$
The numbers $\beta_{i,j}=\beta_{i,j}(S/I)$, where $i,j\geq0$, are called the \emph{graded Betti numbers} of $I$, which count the elements of degree $j$ in a minimal generator of ($i+1$)-th \emph{syzygy}:~ $\mathrm{Syz}_{i+1}(S/I)=\ker~(F_i\rightarrow F_{i-1})$. Let $t_i$ denote the maximal shifts in the minimal graded free $S$-resolution of $S/I$, namely
\[t_i=t_i(S/I):= \max(j:\ \beta_{i,j}(S/I)\neq 0).\]
We say that $I$ satisfies the \emph{subadditivity condition} if $t_{a+b}\leq t_a+t_b$, for all $a,b\geq0$ and $a+b\leq p$, where $p$ is the projective dimension of $I$.

It is known that graded ideals may not satisfy the subadditivity condition as shown by the counter example in \cite[Section 5.4]{Conca-Sub}. However, no counter examples are known for monomial ideals.
For edge ideals of graphs the inequality $t_{a+1}\leq t_a+t_1$ was shown by
Fern{\'a}ndez-Ramos and Gimenez \cite[Theorem 4.1]{Oscar}. The same inequality has been shown later
for any monomial ideal \cite[Corollary 4]{Herzog-Srinivasan} by Srinivasan and Herzog. Yazdan Pour independently proved the same result in \cite[Corollary 3.5]{Yazdan}. Bigdeli and Herzog proved the subadditivity condition, when $I$ is the edge ideal of chordal graph or whisker graph\cite[Theorem 1]{mina-herzog}. Some more results regarding subadditivity have been obtained by Khoury and Srinivasan \cite[Theorem 2.3]{Khoury-Srin}, the author and Nevo
\cite[Theorem 1.3]{Abed-Eran} and Faridi \cite[Theorem 3.7]{Faridi}.

In this paper we prove in Theorem \ref{ab4} that if $I$ is a monomial ideal of $S$ and $a\geq b-1\geq0$ are integers such that $a+b\leq\mathrm{proj~dim}(S/I)$, then $$t_{a+b}\leq t_a+t_1+t_2+\cdots+t_b-\frac{b(b-1)}{2},$$
which generalizes the well-known inequality $t_{a+1}\leq t_a+t_1$.

In Theorem \ref{ab3}, we prove that if $I=I_{\Delta}$ where $\Delta$ is a simplicial complex such that $\dim(\Delta)< t_a-a$ or $\dim(\Delta)< t_b-b$, then $t_{a+b}\leq t_a+t_b$.
The proof of both Theorems \ref{ab3} and \ref{ab4} uses a combinatorial topological argument.

In Theorem \ref{ab5}, we prove algebraically, using Taylor resolution, that the subadditivity condition holds when $I$ is a monomial ideal that minimally generated by $m_1,\dots,m_r$ such that $\frac{\lcm(m_1,\dots,m_r)}{\lcm(m_1,\dots,\widehat{m}_i,\dots,m_r)}\notin K$ for all $i$, where $\widehat{m}_i$ means that $m_i$ is omitted.

\section{Preliminaries }
Fix a field $K$.
Let $S=K[x_1,\dots,x_n]$ be the graded polynomial ring with $\deg(x_i)=1$ for all $i$, and $M$ be a graded $S$-module. The integer $\beta_{i,j}^S(M)=\dim_KTor_i^S(M,K)_j$ is called the $(i,j)$\emph{th graded Betti number of} $M$. Note that if $I$ is a graded ideal of $S$, then $\beta_{i+1,j}^S(S/I)=\beta_{i,j}^S(I)$ for all $i,j\geq 0$.

For a simplicial complex $\Delta$ on the vertex set $\Delta_0=[n]=\{1,\dots,n\}$, its \emph{Stanley-Reisner ideal} $I_{\Delta}\subset S$ is the ideal generated by the squarefree monomials $x_F=\prod_{i\in F}x_i$ with $F\notin \Delta$, $F\subset [n]$. The dimension of the face $F$ is $|F|-1$ and the \emph{dimension of $\Delta$} is $\max\{\dim F~:~F\in\Delta\}$.

For $W\subset V$, we write $$\Delta[W]=\{F\in\Delta~:~F\subset W\}$$ for the induced subcomplex of $\Delta$ on $W$.
 We denote by $\beta_i(\Delta)=\dim_K \widetilde{H}_i(\Delta;K)$ the dimension of the $i$-th reduced homology group of $\Delta$ with coefficients in $K$.
The following result is known as Hochster's formula for graded Betti numbers.

\begin{theorem}[Hochster]
Let $\Delta$ be a simplicial complex on $[n]$. Then
$$\beta_{i,i+j}(S/I_{\Delta})=\sum_{W\subset[n],~|W|=i+j}\beta_{j-1}(\Delta[W];K)$$ for all $i,j\geq0$.
\end{theorem}

If $\Delta_1$ and $\Delta_2$ are two subcomplexes of $\Delta$ such that $\Delta=\Delta_1\cup \Delta_2$, then there is a long exact sequence of reduced homologies, called the \emph{Mayer-Vietoris sequence}
$$\cdots \rightarrow \widetilde{H}_i(\Delta_1\cap\Delta_2;K)\rightarrow \widetilde{H}_i(\Delta_1;K)\oplus \widetilde{H}_i(\Delta_2;K)\rightarrow \widetilde{H}_i(\Delta;K)\rightarrow \widetilde{H}_{i-1}(\Delta_1\cap\Delta_2;K)\rightarrow\cdots.$$

\section{The main Theorems}
\begin{lemma}\label{ab1}

Let $\Delta=\Delta_1\cup\cdots\cup\Delta_t$ be a union of subcomplexes. If $$\widetilde{H}_{j-r+1}(\bigcap_{m=1}^{r}\Delta_{i_m};K)=0$$ for all $1\leq i_1<\cdots<i_{r}\leq t$, then $\widetilde{H}_j(\Delta;K)=0$.

\end{lemma}

\begin{proof}

We prove the assertion by induction on $t$. It is trivial for $t=1$. Let $t=2$. Using the Mayer-Vietoris sequence $$\cdots \rightarrow \widetilde{H}_j(\Delta_1;K)\oplus \widetilde{H}_j(\Delta_2;K)\rightarrow \widetilde{H}_j(\Delta;K)\rightarrow \widetilde{H}_{j-1}(\Delta_1\cap\Delta_2;K)\rightarrow\cdots$$ and the assumptions $\widetilde{H}_j(\Delta_1;K)=\widetilde{H}_j(\Delta_2;K)=\widetilde{H}_{j-1}(\Delta_1\cap\Delta_2;K)=0$, we get $\widetilde{H}_j(\Delta;K)=0$.

Let $t>2$. Now, we consider the Mayer-Vietoris sequence
$$\cdots \rightarrow \widetilde{H}_j(\Delta_1;K)\oplus \widetilde{H}_j(\Delta_2\cup\cdots\cup\Delta_t;K)\rightarrow \widetilde{H}_j(\Delta;K)\rightarrow\widetilde{H}_{j-1}((\Delta_1\cap\Delta_2)\cup\cdots\cup(\Delta_1\cap\Delta_t);K)\rightarrow\cdots$$
by induction assumption, we have
$$\widetilde{H}_j(\Delta_2\cup\cdots\cup\Delta_t;K)=\widetilde{H}_{j-1}((\Delta_1\cap\Delta_2)\cup\cdots\cup(\Delta_1\cap\Delta_t);K)=0.$$
Hence $\widetilde{H}_j(\Delta;K)=0$ as desired.

\end{proof}

\begin{proposition}\label{ab2}
Let $\Delta$ be a simplicial complex on the set $[n]$ and $W\subseteq[n]$ such that $|W|=t_a+s+l+1$, where $a\leq\mathrm{proj~dim}(S/I_{\Delta})$, $s\geq0$ and $l\geq1$. Then for all $A\subseteq W$ with $|A|=s'+l$ where $0\leq s'\leq s$, we have
$$\widetilde{H}_{t_a-a+s}(\bigcup_{B\subseteq A:~|B|=l}(\Delta[W\setminus B];K))=0.$$
\end{proposition}

\begin{proof}
We prove the proposition by induction on $l$. Let $l=1$. Note that for all $1\leq r\leq |A|$, we have $\beta_{a,t_a+s+2-r}(S/I_{\Delta})=0$. So $$\widetilde{H}_{t_a-a+s-r+1}(\Delta[W\setminus B_1]\cap\cdots\cap\Delta[W\setminus B_r];K)=0.$$
for all distinct singleton subsets $B_1,\dots,B_r$ of $A$. By Lemma \ref{ab1}, we obtain that $$\widetilde{H}_{t_a-a+s}(\bigcup_{B\subseteq A:~|B|=1}(\Delta[W\setminus B];K))=0.$$
Let $l>1$ and $A=\{b_1,\dots,b_{s'},b_{s'+1},\dots,b_{s'+l}\}$. For $1\leq j\leq s'+1$, denote by $B_j$ the set of all $B\subseteq A$ such that $|B|=l$, $b_j\in B$ and $\{b_1,\dots,b_{j-1}\}\cap B=\emptyset$. For $1\leq j\leq s'+1$, let $\Delta_{j}=\bigcup_{B\in B_j}\Delta[W\setminus B]$. We have that $$\bigcup_{B\subseteq A:~|B|=l}(\Delta[W\setminus B])=\Delta_1\cup\cdots \cup\Delta_{s'+1}.$$
Let $1\leq i_1<\cdots<i_{r}\leq s'+1$. Denote by $W'$ the set $W\setminus\{b_{i_1},\dots,b_{i_r}\}$ and by $A'$ the set $A\setminus\{b_1,\dots,b_{i_r}\}$. Let $\overline{s}=s-r+1$ and $\overline{l}=l-1$. We have that $|W'|=t_a+\overline{s}+\overline{l}+1$ and $$\Delta_{i_1}\cap\cdots\cap\Delta_{i_r}=\bigcup_{B\subseteq A':~|B|=\overline{l}}\Delta[W'\setminus B].$$
By induction hypothesis, we have that $$\widetilde{H}_{t_a-a+\overline{s}}(\Delta_{i_1}\cap\cdots\cap\Delta_{i_r};K)=0.$$
By Lemma \ref{ab1} we get $$\widetilde{H}_{t_a-a+s}(\bigcup_{B\subseteq A:~|B|=l}(\Delta[W\setminus B];K))=0.$$
\end{proof}
Now, we prove the main results.

\begin{theorem}\label{ab3}
Let $\Delta$ be a simplicial complex on the set $[n]$ and $a,b$ are non negative integers such that $a,b\leq\mathrm{proj~dim}(S/I_{\Delta})$. If $\dim(\Delta)< t_a-a$ or $\dim(\Delta)< t_b-b$, then $t_{a+b}\leq t_a+t_b$.
\end{theorem}

\begin{proof}
Without loss of generality, assume that $\dim(\Delta)< t_b-b$. We have to show that $\beta_{a+b,t_a+t_b+r+1}(S/I_{\Delta})=0$ for all integer $r\geq0$. Let $W\subseteq[n]$ of size $t_a+t_b+r+1$ and $A$ be any subset of $W$ with $|A|=t_b+r$, where $r\geq0$. Since $\Delta$ does not contains a face of dimension $t_b-b$, it follows that $$\Delta[W]=\bigcup_{B\subseteq A:~|B|=b}\Delta[W\setminus B].$$
By \ref{ab2} (taking $s=s'=t_b-b+r$ and $l=b$), it follows that $$\widetilde{H}_{t_a-a+t_b-b+r}(\Delta[W];K)=0.$$
\end{proof}

\begin{theorem}\label{ab4}
If $I$ is a monomial ideal of $S$, $b\geq 1$ and $a\geq b-1$ are integers such that $a+b\leq\mathrm{proj~dim}(S/I)$, then $$t_{a+b}\leq t_a+t_1+t_2+\cdots+t_b-\frac{b(b-1)}{2}.$$
\end{theorem}

\begin{proof}
By polarization, we may assume that $I$ is a squarefree ideal, and so $I=I_{\Delta}$, where $\Delta$ is a simplicial complex on $[n']$. First, we prove that
\begin{equation}\label{eq1}
t_{c+1}\leq t_{c}+t_d-d+1,
\end{equation}
for all $d\geq1$ and $d-1\leq c\leq\mathrm{proj~dim}(S/I)$. Assume on contrary, that $\beta_{c+1,t_c+t_d-d+r+2}(S/I)\neq0$ for some $r\geq0$. It follows that there exists a subset $W$ of $[n']$ so that $|W|=t_c+t_d-d+r+2$ and $\widetilde{H}_{t_c-c+t_d-d+r}(\Delta[W];K)\neq0$. In particular, $\mathrm{proj~dim}(S/I_{\Delta[W]})\geq c+1$.\\ If $\deg(m)\geq t_d-d+r+2$ for all minimal generator $m$ of $I_{\Delta[W]}$, then $\beta_{1,x}(S/I_{\Delta[W]})=0$ for all $x\leq t_d-d+r+1$. So $\beta_{d,x}(S/I_{\Delta[W]})=0$ for all $x\leq t_d-d+r+d=t_d+r$. It follows that $\mathrm{proj~dim}(S/I_{\Delta[W]})< d\leq c+1$. This is a contradiction.

We obtain that there is a subset $A$ of $W$ with $|A|=t_d-d+r+1$ and $$\Delta[W]=\bigcup_{B\subseteq A:~|B|=1}\Delta[W\setminus B].$$
By \ref{ab2} (taking $s=s'=t_d-d+r$ and $l=1$), it follows that $$\widetilde{H}_{t_c-c+t_d-d+r}(\Delta[W];K)=0,$$
a contradiction.

Now, we prove by induction on $1\leq b'\leq b$ that $$t_{a+b'}\leq t_a+t_1+t_2+\cdots+t_{b'}-\frac{b'(b'-1)}{2}.$$ The case $b'=1$ follows by (\ref{eq1}). Let $b'>1$. By the induction hypothesis, we have
\begin{equation*}
\begin{split}
t_{a+b'}=t_{(a+b'-1)+1} & \leq t_{a+b'-1}+t_{b'}-b'+1 \\
 & \leq t_a+t_1+\cdots+t_{b'-1}-\frac{(b'-1)(b'-2)}{2}+t_{b'}-b'+1\\
 &=t_a+t_1+t_2+\cdots+t_{b'}-\frac{b'(b'-1)}{2}.
\end{split}
\end{equation*}
\end{proof}

\section{The Taylor resolution}
Let $I$ be a monomial ideal with $G(I)=\{m_1,\dots,m_r\}$. For a subset $F$ of $G(I)$, set $\lcm(F)=\lcm\{m_i:~m_i\in F\}$ and define a formal symbol $[F]$ with multidegree equal to $\lcm(F)$. Let $T_0=S$ and for each $i\geq1$, let $T_i$ be the free $S$-module with basis $\{[F]:~|F|=i\}$. Note that $T_i$ is a multigraded $S$-module.

Let $\phi_0:T_0\rightarrow S/I$ be the canonical homomorphism and define the multigraded differential $\phi_{i}:T_i\rightarrow T_{i-1}$ by $$[F]\longmapsto\sum_{k=1}^{i}(-1)^{k-1}\cdot\frac{\lcm(F)}{\lcm(F\setminus\{m_{j_k}\})}\cdot[F\setminus\{m_{j_k}\}]$$
where $F=\{m_{j_1},\dots,m_{j_i}\}$, written with the indices in increasing order.

The free resolution $$\mathbb{T}:~0\longrightarrow T_r\overset{\phi_r}{\longrightarrow}\cdots\longrightarrow T_1\overset{\phi_1}{\longrightarrow} T_0\overset{\phi_0}{\longrightarrow} S/I\longrightarrow 0$$
is called the \emph{Taylor resolution} of $I$.

We denote by $H_a(\overline{\mathbb{T}})$ the a-th homology group, of the chain complex
$$\overline{\mathbb{T}}:~0\longrightarrow T_r\otimes_{S}K\overset{\phi_r\otimes_{S}K}{\longrightarrow}\cdots\longrightarrow T_1\otimes_{S}K\overset{\phi_1\otimes_{S}K}{\longrightarrow} T_0\otimes_{S}K\longrightarrow 0.$$
Note that $\beta_{a,b}(S/I)=\dim_K(H_a(\overline{\mathbb{T}}))_b$ and for all $0\leq j\leq r$ $$T_j\otimes_{S}K\cong T_j\otimes_{S}S/M\cong T_j/MT_j=\overline{T}_j$$
where $M$ is the maximal ideal $(x_1,\dots,x_n)$ of $S$.

\begin{proposition}\label{ab5}
Let $I$ be a monomial ideal with $G(I)=\{m_1,\dots,m_r\}$. If there exists $0\neq[F]\in H_{a+b}(\overline{\mathbb{T}})$ such that $\deg(\lcm(F))=t_{a+b}$, then $t_{a+b}\leq t_a+t_b$.
\end{proposition}

\begin{proof}
With out loos of generality, assume that $F=\{w_1,\dots,w_{a+b}\}$ and $$G(I)=\{m_1,\dots,m_{r-(a+b)},w_1,\dots,w_{a+b}\}.$$ Set $F_1=\{w_1,\dots,w_a\}$ and $F_2=\{w_{a+1},\dots,w_{a+b}\}$. Since $[F]\in \ker \overline{\phi}_{a+b}$, it follows that $\frac{\lcm(F)}{\lcm(G)}\notin K$, for all $G\subsetneq F$. So we have that $[F_1]\in \ker \overline{\phi}_{a}$ and $[F_2]\in \ker \overline{\phi}_{b}$. If $[F_1]\in \Ima \overline{\phi}_{a+1}$, then there exist $[F_{1,1}],\dots,[F_{1,s}]\in \overline{T}_{a+1}$ and $\alpha_1,\dots,\alpha_s\in K$ such that $$[F_1]=\alpha_1\overline{\phi}_{a+1}[F_{1,1}]+\cdots+\alpha_s\overline{\phi}_{a+1}[F_{1,s}].$$
Note that $\lcm(F_1)=\lcm(F_{1,j})$ for all $1\leq j\leq s$, so $F_{1,j}\cap F_2=\emptyset$ for all $1\leq j\leq s$. For all $1\leq j\leq s$ let $\widehat{F}_{1,j}=F_{1,j}\cup F_2$ and assume that the elements of $\widehat{F}_{1,j}$ are ordered as the order of $G(I)$. We obtain that $[F]=\alpha_1\overline{\phi}_{a+b+1}[\widehat{F}_{1,1}]+\cdots+\alpha_s\overline{\phi}_{a+b+1}[\widehat{F}_{1,s}]$,
which contradicts to the fact that $[F]\notin \Ima\overline{\phi}_{a+b+1}$. Similarly, $[F_2]\notin \Ima \overline{\phi}_{b+1}$.

It follows that $$t_{a+b}=\deg(\lcm(F))\leq \deg(\lcm(F_1))+\deg(\lcm(F_2))\leq t_a+t_b.$$
\end{proof}

\begin{corollary}\label{ab6}
Let $I$ be a monomial ideal with $G(I)=\{m_1,\dots,m_r\}$. If $\frac{\lcm(m_1,\dots,m_r)}{\lcm(m_1,\dots,\widehat{m}_i,\dots,m_r)}\notin K$ for all $i$, where $\widehat{m}_i$ means that $m_i$ is omitted, then $t_{a+b}\leq t_a+t_b$.
\end{corollary}

\begin{proof}
Let $[F]$ be a generator of $T_{a+b}$ with $\deg(\lcm(F))=t_{a+b}$. By the assumption, it follows that $[F]\in \ker \overline{\phi}_{a+b}$ and $[F]\notin \Ima \overline{\phi}_{a+b+1}$. Hence the assertion follows from Proposition \ref{ab5}.
\end{proof}

\begin{example}
Let $S=K[a,b,c]$ and $I$ be the ideal of $S$ which generated by $G(I)=\{a^4bc,b^3c^2,c^5a^3\}$. Not that
$$\frac{\lcm(G(I))}{\lcm(b^3c^2,c^5a^3)}=a,~\frac{\lcm(G(I))}{\lcm(a^4bc,c^5a^3)}=b^2,~\frac{\lcm(G(I))}{\lcm(a^4bc,b^3c^2)}=c^3.$$
So by Corollary \ref{ab6}, $I$ satisfies the subadditivity condition.
\end{example}



\begin{thebibliography}{99}

   \bibitem{Oscar} O. Fern{\'a}ndez-Ramos, P. Gimenez, \emph{Regularity 3 in edge ideals associated to bipartite
graphs}, J. Algebraic Combin. 39, 919-937 (2014).

 \bibitem{Herzog-Srinivasan} J. Herzog, H. Srinivasan, \emph{A note on the subadditivity problem for maximal shifts in free
resolutions}, Commutative Algebra and Noncommutative Algebraic Geometry, II. MSRI
Publications, vol. 68, 245-250 (2015).


  \bibitem{Conca-Sub} L. L. Avramov, A. Conca, S. Iyengar, \emph{Subadditivity of syzygies of Koszul algebras}, Math. Ann. 361, 511-534 (2015).

   \bibitem{Yazdan} Ali Akbar Yazdan Pour, \emph{Candidates for nonzero Betti numbers of monomila ideals}, Comm. Algebra, no.4, 1488-1492 (2017).

  \bibitem{Khoury-Srin} S. E. Khoury, H. Srinivasan, \emph{A note on the subadditivity of syzygies}. J. Algebra Appl. 1750177 (2016).

  \bibitem{Abed-Eran} A. Abedelfatah, E. Nevo, \emph{On vanishing patterns in $j$-strands of edge ideals}, J. Algebraic Comb. 46(2), 287-295 (2017).

     \bibitem{Faridi} S. Faridi, Lattice complements and the subadditivity of syzygies of simplicial forests (2016),
         http://arxiv.org/abs/1605.07727

    \bibitem{mina-herzog} M. Bigdeli and J. Herzog, \emph{Betti diagrams with special shape}, Homological and Computational Methods in Commutative algebra, Springer INdAM Series 20, 33-52 (2017)

        \end{thebibliography}
\end{document}